\documentclass[reqno]{amsart}
\usepackage{amsfonts,amsmath,amssymb,amsrefs,dsfont}
\usepackage{color}
\usepackage[utf8]{inputenc}  
\usepackage[T1]{fontenc}  
\usepackage{graphicx} 
\numberwithin{equation}{section}

\usepackage[colorinlistoftodos]{todonotes}

\newtheorem{thm}{Theorem}[section]
\newtheorem*{thm*}{Theorem}

\newtheorem*{cor*}{Corollary}
\newtheorem{prop}{Proposition}[section]

\newtheorem{Step}{Step}[section]

\begin{document}
\title[Non-existence of extremals for the Adimurthi-Druet inequality]{Non-existence of extremals for the Adimurthi-Druet inequality}
\author{Gabriele Mancini}
\address[Gabriele Mancini]{Universit\`a degli Studi di Padova, Dipartimento di Matematica Tullio Levi-Civita, Via Trieste, 63, 35121 Padova}
\email{gabriele.mancini@math.unipd.it}

\author{Pierre-Damien Thizy}
\address[Pierre-Damien Thizy]{Universit\'e Claude Bernard Lyon 1, CNRS UMR 5208, Institut Camille Jordan, 43 blvd. du 11 novembre 1918, F-69622 Villeurbanne cedex, France}
\email{pierre-damien.thizy@univ-lyon1.fr}
\date{November 2017}

\begin{abstract}
The Adimurthi-Druet \cite{AdimurthiDruet} inequality is an improvement of the standard Moser-Trudinger inequality by adding a $L^2$-type perturbation, quantified by $\alpha\in[0,\lambda_1)$, where $\lambda_1$ is the first Dirichlet eigenvalue of $\Delta$ on a smooth bounded domain. It is known \cite{CarlesonChang,Flucher,StruweCrit,LuYang} that this  inequality admits extremal functions, when the perturbation parameter $\alpha$ is small. By contrast, we prove here that the  Adimurthi-Druet inequality does not admit any extremal, when the perturbation parameter $\alpha$ approaches $\lambda_1$.  Our result is based on sharp expansions of the Dirichlet energy for blowing sequences of solutions of the corresponding Euler-Lagrange equation, which take into account the fact that the problem becomes singular as $\alpha\to \lambda_1$. 
\end{abstract}

\maketitle
\section{Introduction}
Let $\Omega$ be a smooth bounded domain of $\mathbb{R}^2$. We let $H^1_0$ be the usual Sobolev and Hilbert space of functions in $\Omega$, endowed with the scalar product
$$\langle u,v\rangle_{H^1_0}=\int_\Omega \nabla u.\nabla v~ dy\,, $$
and with the associated norm denoted by $\|\cdot\|_{H^1_0}$.
For all $\alpha\ge 0$, we let $C_\alpha(\Omega)$ be given by
\begin{equation}\label{AdimDruIneq1}
C_\alpha(\Omega)=\sup_{\left\{u\in H^1_0\text{ s.t. }\|\nabla u\|_2=1\right\}}\int_\Omega \exp\left(4\pi u(y)^2\left(1+\alpha \|u\|_2^2 \right) \right) dy\,.
\end{equation}
Then, the Adimurthi-Druet \cite{AdimurthiDruet} inequality claims that
$$C_\alpha(\Omega)<+\infty\Leftrightarrow \alpha<\lambda_1\,, $$
where $\lambda_1>0$ is the first eigenvalue of $\Delta=-\partial_{xx}-\partial_{yy}$ in $\Omega$ with zero Dirichlet boundary condition on $\partial\Omega$.\\
\indent While the existence of an extremal function for $\alpha=0$, i.e. for the standard Moser-Trudinger inequality, was obtained by Carleson-Chang \cite{CarlesonChang}, Struwe \cite{StruweCrit} and Flucher \cite{Flucher}, Yang and Lu \cite{LuYang} were able to prove that there exists an extremal function for \eqref{AdimDruIneq1} for all $\alpha\ge 0$ sufficiently close to $0$. More recently, still concerning the original Adimurthi-Druet inequality \eqref{AdimDruIneq1}, it was explained  in Yang \cite{YangJDE} that the existence of extremals for more general $\alpha$'s closer to $\lambda_1$ is left open. We prove here that, surprisingly, there is no extremal function for \eqref{AdimDruIneq1} for all $\alpha<\lambda_1$ sufficiently close to $\lambda_1$. Then, our main result is stated as follows.
\begin{thm}[Non existence of extremals]\label{MainThm}
Let $\Omega$ be a smooth, bounded and connected domain of $\mathbb{R}^2$. Let $\lambda_1>0$ be the first eigenvalue of $\Delta$ with zero Dirichlet boundary condition. Then there exists $\alpha_0\in(0,\lambda_1)$ such that, for all $\alpha\in[\alpha_0,\lambda_1)$, there is no extremal function for \eqref{AdimDruIneq1}.
\end{thm}

The proof of Theorem \ref{MainThm} relies on the recent progresses concerning the blow-up analysis of Moser-Trudinger equations (see \cites{MalchMartJEMS,DruThiI}). The difficulty in this problem is a cancellation of the first terms in the Dirichlet energy expansions of Section \ref{SectProofThm}, which enforces to carry out in Section \ref{SectMainProp} a very precise blow-up analysis. For instance, the estimates obtained in \cite{AdimurthiDruet} and \cite{LuYang} are far from being sufficient to conclude here. Note that a similar cancellation was already observed by Martinazzi-Mancini \cite{MartMan} in the radial case, namely when $\Omega$ is the unit disk $\mathbb{D}^2$ of $\mathbb{R}^2$. Even in this more particular case, the authors had to carry out a very careful blow-up analysis of the next lower order terms in order to conclude. To be able to deal with the general (non necessarily radial) situation, we use here the techniques developed in Druet-Thizy \cite{DruThiI}. But, a new additional serious difficulty here is that the problem becomes singular when $\alpha$ gets close to $\lambda_1$. By \textit{singular}, we mean here that the kernel of the operator obtained by linearizing the limiting equation at $0$ does not only {contain} the zero function. Here (see \eqref{St1}, \eqref{C} and \eqref{LambdaTo0}), this operator is $\Delta-\lambda_1$ with zero Dirichlet boundary condition and  we have to compute carefully (see Step \ref{StGreenFun}) what happens in its kernel.

\indent { As already observed by Del Pino-Musso-Ruf \cite{DelPNewSol} in the non-singular case, the critical exponential non-linearity $\exp(u^2)$ in dimension $2$ is more difficult to handle than the Sobolev critical non-linearity $u^{\frac{n+2}{n-2}}$ in higher dimensions $n>2$, and getting sharp energy expansions of positive blow-up solutions reveals to be delicate in this case. Besides, even for Sobolev critical problems in higher dimensions, understanding the behavior of positive blow-up  solutions turned out to be very challenging in the singular case. This difficulty was overcome while solving Lin-Ni's conjecture (see Druet-Robert-Wei \cite{LinNiAMS}, Rey-Wei \cite{ReyWei} and Wei-Xu-Yang \cite{WeiXuYang} and the references therein), where the limiting linearized operator is $\Delta$ with zero Neumann boundary condition, whose kernel is the set of the constant functions.}

\indent As far as we know, Theorem \ref{MainThm} is the first result proving the non-existence of extremals for an explicit Moser-Trudinger type inequality with critical exponent on bounded domains. Indeed, similar results had so far been proven only for implicit perturbations of the Moser-Trudinger inequality \cite{Pruss}, or for sub-critical inequalities on $\mathbb R^2$ \cite{Ishiwata}, where  blow-up of maximizing sequences cannot occur.

The paper is organized as follows. Theorem \ref{MainThm} is proved in Section \ref{SectProofThm}. This proof relies on the key energy estimates of Proposition \ref{MainProp}, whose proof is given in Section \ref{SectMainProp}.

\section*{Acknowledgements} 
This work was initiated at the department of Mathematics and Computer Science of the university of Basel, where the first author was employed until September 2017.  
The authors warmly thank Luca Martinazzi for having invited the second author in Basel and for many fruitful discussions on these topics. The first author was supported by Swiss National Science Foundation, projects nr. PP00P2-144669 and PP00P2-170588/1.

\section{Proof of Theorem \ref{MainThm}}\label{SectProofThm}
Assume by contradiction that there exists a sequence $(\alpha_i)_i$ such that $\alpha_i\to \lambda_1^-$ and such that there exists an extremal function $u_{\alpha_i}\ge 0$ for $C_{\alpha_i}(\Omega)$. For simplicity, we drop the indexes $i$'s. Then the $u_\alpha$'s satisfy
\begin{equation}\label{ELEquation}
\begin{cases}
&\Delta u_\alpha=A_\alpha u_\alpha+2\beta_\alpha \lambda_\alpha u_\alpha\exp(\beta_\alpha u_\alpha^2)\text{ in }\Omega\,,\quad u_\alpha=0\text{ in }\partial\Omega\,,\\
&\|\nabla u_\alpha\|_2^2=1\,,\\
&\beta_\alpha=4\pi\left(1+\alpha \|u_\alpha\|_2^2 \right)\,,\\
&A_\alpha=\frac{\alpha}{1+2\alpha\|u_\alpha\|_{2}^2}<\lambda_1\,,
\end{cases}
\end{equation}
for some positive $\lambda_\alpha$'s, and in particular, the $u_\alpha$'s are smooth. Indeed, the Moser-Trundinger inequality gives that
$$u\in H^1_0\implies \exp(u^2)\in L^p\,, $$
for all $1\le p<+\infty$, and then standard elliptic theory applies. Since $C_{\lambda_1}(\Omega)=+\infty$, we get that
\begin{equation}\label{CAlpha}
C_\alpha(\Omega)\to +\infty
\end{equation}
as $\alpha\to \lambda_1^-$, by the monotone convergence theorem. Then, by Lions \cite[Theorem I.6]{Lions}, we have that, up to a subsequence,
\begin{equation}\label{Step0}
u_\alpha\rightharpoonup 0 \text{ in }H^1_0\,,\quad u_\alpha\to 0 \text{ in }L^p\text{ for all }p<+\infty\,, \quad\|u_\alpha\|_{L^\infty}\to +\infty
\end{equation}
and thus that
\begin{equation}\label{St1}
\beta_\alpha\to 4\pi\,\text{ and } A_\alpha\to \lambda_1^-\,,
\end{equation}
as $\alpha\to \lambda_1^-$. Now we rephrase everything in terms of 
\begin{equation}\label{V}
v_\alpha:=\sqrt{\beta_\alpha} u_\alpha\,.
\end{equation}
We have that
\begin{equation}\label{C}
\begin{cases}
&\Delta v_\alpha=v_\alpha\left( A_\alpha+\Lambda_\alpha \exp(v_\alpha^2)\right)\text{ in }\Omega\,,\quad v_\alpha=0\text{ in }\partial\Omega\,,\\
&\beta_\alpha=4\pi\left(1+\frac{\alpha}{\beta_\alpha} \|v_\alpha\|_2^2 \right)\,,\\
&A_\alpha=\frac{\alpha}{1+2\frac{\alpha}{\beta_\alpha}\|v_\alpha\|_{2}^2}(<\alpha<\lambda_1)\,,
\end{cases}
\end{equation}
where $\Lambda_\alpha=2\beta_\alpha\lambda_\alpha>0$. Moreover, $\|\nabla u_\alpha\|^2_2=1$ implies
\begin{equation}\label{D}
\|\nabla v_\alpha\|_{2}^2=\beta_\alpha\,.
\end{equation}
We also get that $\|v_\alpha\|_2\to 0$ as $\alpha\to \lambda_1^-$ and the second line of \eqref{C} implies
\begin{equation}\label{E}
\begin{split}
\beta_\alpha&=2\pi\left(1+\sqrt{1+\frac{\alpha\int_\Omega v_\alpha^2 dy}{\pi}} \right)\\
&=4\pi \left(1+\frac{\alpha\int_\Omega v_\alpha^2 dy}{4\pi}-\frac{\alpha^2\left(\int_\Omega v_\alpha^2 dy\right)^2}{16\pi^2}+O\left(\left(\int_\Omega v_\alpha^2 dy \right)^3 \right) \right)\,.
\end{split}
\end{equation}
 
 Now, we have that
$$\int_\Omega \exp(v_\alpha^2) dy=\int_\Omega \exp\left(\beta_\alpha u_\alpha^2 \right) dy=C_\alpha(\Omega)\to +\infty \,,$$
and, independently, that
$$\Lambda_\alpha \int_\Omega v_\alpha^2 \exp(v_\alpha^2) dy=\int_\Omega |\nabla v_\alpha|^2 dy-A_\alpha\int_\Omega v_\alpha^2 dy=4\pi+o(1)\,,$$
so that there must be the case that 
\begin{equation}\label{LambdaTo0}
\Lambda_\alpha\to 0
\end{equation}
 as $\alpha\to \lambda_1^-$,  since $e^{t}\le 1+t e^t$ for $t\ge 0$.  We are now in position to use Proposition \ref{MainProp} below: we have that
\begin{equation}\label{F}
\|\nabla v_\alpha\|_2^2=4\pi\Big(1+\frac{A_\alpha}{4\pi}\int_\Omega v_\alpha^2 dy+o\left(\left(\int_\Omega v_\alpha^2 dy\right)^2 \right)\Big)
\end{equation}
as $\alpha\to \lambda_1^-$. Then, expanding the third line of \eqref{C}, we get
\begin{equation}\label{G}
A_\alpha=\alpha-\frac{\alpha^2\int_\Omega v_\alpha^2 dy}{2\pi}+O\left(\left( \int_\Omega v_\alpha^2 dy\right)^2\right)\,.
\end{equation}
Now, \eqref{F} and \eqref{G} give
\begin{equation}\label{H}
\beta_\alpha=4\pi\left(1+\frac{\alpha}{4\pi}\int_\Omega v_\alpha^2 dy-\frac{\alpha^2}{8\pi^2}\left(\int_\Omega v_\alpha^2 dy\right)^2+o\left(\left(\int_\Omega v_\alpha^2 dy\right)^2 \right) \right)\,.
\end{equation}
But \eqref{E} and \eqref{H} have to match, then we get
\begin{equation} \label{I}
-\frac{\lambda_1^2}{16\pi^2}\left(\int_\Omega v_\alpha^2 dy\right)^2=o\left(\left(\int_\Omega v_\alpha^2 dy\right)^2 \right)\,,
\end{equation}
which is the contradiction we look for.

\section{Blow-up analysis on \eqref{C}}\label{SectMainProp}
\begin{prop}\label{MainProp}
Let $(v_\alpha)_\alpha$ be a sequence of smooth solutions of 
\begin{equation}\label{CBis}
\begin{cases}
&\Delta v_\alpha=v_\alpha\left( A_\alpha+\Lambda_\alpha \exp(v_\alpha^2)\right)\,,\quad v_\alpha>0\text{ in }\Omega\,,\\
&v_\alpha=0\text{ in }\partial\Omega\,,
\end{cases}
\end{equation}
 for $A_\alpha\in [0,\lambda_1)$ and $\Lambda_\alpha>0$, for all $\alpha$ slightly smaller than $\lambda_1$. We let $\beta_\alpha>0$ be given by \eqref{D} and we assume that \eqref{St1} and \eqref{LambdaTo0} hold true. We also assume that the $v_\alpha$'s blow-up, namely that 
 \begin{equation}\label{DefGamma}
\gamma_\alpha:=\max_\Omega v_\alpha=v_\alpha(x_\alpha)\to +\infty\,,
\end{equation} 
as $\alpha\to \lambda_1^-$, for $x_\alpha\in \Omega$. Then, we have that
\begin{equation}\label{F0}
\gamma_\alpha^2{\int_\Omega v_\alpha^2 dy}\to +\infty\,,
\end{equation}
that
\begin{equation}\label{AccurateLambdaEstim}
\Lambda_\alpha=o\left(\frac{1}{\gamma_\alpha^2} \right)
\end{equation}
and that
\eqref{F} hold true as $\alpha\to \lambda_1^-$.
\end{prop}
Note that \eqref{F0} and \eqref{AccurateLambdaEstim} (proved in Step \ref{StGreenFun}) are specific to our singular case $A_\alpha\to \lambda_1^-$: they would not hold true if the limit of the $A_\alpha$'s were in $[0,\lambda_1)$.\\
\indent Now we turn to the proof of this result. In order to prove Proposition \ref{MainProp}, we study the asymptotic behavior of the $v_\alpha$'s as $\alpha\to \lambda_1^-$. We make the assumptions of Proposition \ref{MainProp}. First, by these assumptions on $(\Lambda_\alpha)_\alpha$ and $(A_\alpha)_\alpha$, the family $(f_\alpha)_\alpha$ of functions, given by $$f_\alpha(t)=t(A_\alpha+\Lambda_\alpha \exp(t^2))\,,$$
is of \textit{uniform critical growth} in the sense of \cite[Definition 1]{DruetDuke}. Also, as in  \cite{DruetDuke} (see also the original argument in \cite{AdimurthiStruwe}),  if $\mu_\alpha$ is given by
\begin{equation}\label{ScalRel}
\mu_\alpha^{-2}:=\frac{\Lambda_\alpha}{4} \gamma_\alpha^2 \exp(\gamma_\alpha^2)\to +\infty\,,
\end{equation}
then there exists a sequence of positive numbers $(R_\alpha)_\alpha$ such that $R_\alpha\to +\infty$, $R_\alpha\mu_\alpha\ll d(x_\alpha,\partial \Omega)$, and 
\begin{equation}\label{FirstResc}
\left\|\gamma_\alpha\left(\gamma_\alpha-v_\alpha\left(x_\alpha+\mu_\alpha\cdot \right) \right)-T_0\right\|_{C^2({B_0(R_\alpha)})}\to 0
\end{equation}
as $\alpha\to \lambda_1^-$, where $T_0:=\log\left(1+|\cdot|^2 \right)$. We recall that $T_0$ solves  the  Liouville equation 
\begin{equation}\label{EqT0}
\Delta T_0 =4\exp(-2T_0) 
\end{equation}
in $\mathbb R^2$. Note that \eqref{CBis}, \eqref{DefGamma} and $\|v_\alpha\|_{H^1_0}=O(1)$ imply that
\begin{equation}\label{IntermComput}
\Lambda_\alpha \exp(\gamma_\alpha^2)\to +\infty\,,
\end{equation}
as $\alpha\to \lambda_1^-$. Moreover, the PDE in \eqref{CBis} is autonomous and  the $f_\alpha$'s are increasing in $[0,+\infty)$. Therefore, as pointed out in \cite{AdimurthiDruet}, the arguments in de Figueiredo-Lions-Nussbaum \cite{deFigLions} and Han \cite{Han} give that the $x_\alpha$'s do not go to the boundary {of $\Omega$}. Then, up to a subsequence,
\begin{equation}\label{NotToBdry}
x_\alpha\to \bar{x}
\end{equation}
as $\alpha\to \lambda_1^-$, for some $\bar{x}\in\Omega$. Let $B_\alpha$ be the radially symmetric solution around $x_\alpha$ of
\begin{equation}\label{EqB}
\begin{cases}
&\Delta B_\alpha=B_\alpha\left(A_\alpha+\Lambda_\alpha \exp\left(B_\alpha^2 \right) \right)\,,\\
&B_\alpha(x_\alpha)=\gamma_\alpha\,.
\end{cases}
\end{equation}
 Let $\bar{v}_\alpha$ be given by 
\begin{equation}\label{vBar}
\bar{v}_\alpha(z)=\frac{1}{2\pi |x_\alpha-z|}\int_{\partial B_{x_\alpha}(|x_\alpha-z|)} v_\alpha(y)~ d\sigma(y)
\end{equation}
for all $z\in\Omega\backslash\{x_\alpha\}$ and $\bar{v}_\alpha(x_\alpha)=v_\alpha(x_\alpha)$. Also we let $t_\alpha$ be given by 
\begin{equation}\label{DefTAlpha}
t_\alpha(y)=\log\left(1+\frac{|y-x_\alpha|^2}{\mu_\alpha^2} \right)= T_0\left(\frac{{y}-x_\alpha}{\mu_\alpha}\right)\,.
\end{equation}
 By abuse of notations, we will write sometimes $B_\alpha(r)$, $t_\alpha(r)$ or $\bar{v}_\alpha(r)$ instead of $B_\alpha(z)$, $t_\alpha(z)$ or $\bar{v}_\alpha(z)$ respectively, for $|z-x_\alpha|=r$. For any $\delta\in(0,1)$, we let $r_{\alpha,\delta}>0$ be given by 
\begin{equation}\label{DefRAlphaDelta}
t_\alpha(r_{\alpha,\delta})=\delta \gamma_\alpha^2\,.
\end{equation} 
Observe that \eqref{DefRAlphaDelta} implies 
\begin{equation}\label{RalphaandMualpha}
r_{\alpha,\delta}^2 = \mu_\alpha^2 \exp{\left(\delta\gamma_\alpha^2 +o(1)\right)} \gg \mu_\alpha^2,
\end{equation}
as $\alpha\to \lambda_1^-$. At last, by \cite[Proposition 2]{DruetDuke}, there exists $D_0>0$ such that
\begin{equation}\label{WeakGradEst}
|\cdot-x_\alpha||\nabla v_\alpha| v_\alpha\le D_0\text{ in }\Omega
\end{equation} 
for all $0<\lambda_1-\alpha\ll 1$. The first rather elementary step is as follows.

\begin{Step}\label{Elementary}
As $\alpha\to \lambda_1^-$, we have that 
\begin{equation}\label{LambdaNotTooSmall}
 \left|\log \Lambda_\alpha\right|= o(\gamma_\alpha^2).
\end{equation}
Moreover, for all $\delta\in(0,1)$ and all sequences $(z_\alpha)_\alpha$ of points $z_\alpha\in B_{x_\alpha}(r_{\alpha,\delta})$, we have that
 \begin{equation}\label{MinorVB}
   v_\alpha(z_\alpha)\ge \gamma_\alpha(1-\delta+o(1)),
\end{equation}
and in particular $r_{\alpha,\delta}<d(x_\alpha,\partial\Omega)$. 
\end{Step}

\begin{proof}[Proof of Step \ref{Elementary}] 
Let $R_0>0$ be such that $\overline{\Omega}\subset B_{\bar{x}}(R_0)$. Let also $\Omega_\alpha$ be given by $\Omega_\alpha=B_{x_\alpha}(R_0)\backslash B_{x_\alpha}(\mu_\alpha)$. We extend $v_\alpha$ by $0$ outside $\Omega$. Let $V_\alpha$ be the unique harmonic function in $\Omega_\alpha$ such that $V_\alpha=v_\alpha$ in $\partial \Omega_\alpha$.  Then, by construction of $V_\alpha$, we know that
\begin{equation}\label{El1}
\int_{\Omega_\alpha}|\nabla V_\alpha|^2 dy\le \int_{\Omega_\alpha} |\nabla v_\alpha|^2 dy\,,
\end{equation}
for all $\alpha$. Let now $\check{A}_\alpha>0$ be such that $\Psi_\alpha:=\check{A}_\alpha \log\frac{R_0}{|\cdot-x_\alpha|}$ and $\gamma_\alpha-\frac{t_\alpha}{\gamma_\alpha}$ coincide on $\partial B_{x_\alpha}(\mu_\alpha)$. Then, we easily get from \eqref{ScalRel} that
\begin{equation}\label{El2}
\check{A}_\alpha 
= \frac{\gamma_\alpha^2-\log 2}{\gamma_\alpha \log\frac{R_0}{\mu_\alpha}}=\frac{\gamma_\alpha(1+o(1))}{\log\frac{1}{\mu_\alpha}}\,.
\end{equation}
By \eqref{FirstResc} and elliptic estimates, we get that
\begin{equation}\label{El3}
|\nabla V_\alpha-\nabla \Psi_\alpha|\le o\left(\frac{1}{\gamma_\alpha|\cdot-x_\alpha|} \right)\text{ in }\Omega_\alpha
\end{equation} 
for all $0<\lambda_1-\alpha\ll 1$. Then, we get from \eqref{El2} and \eqref{El3} that 
\begin{equation}\label{El4}
\begin{split}
\int_{\Omega_\alpha} |\nabla V_\alpha|^2 dy &=\pi \check{A}_\alpha^2\log\frac{1}{\mu_\alpha^2}(1+o(1))\\
&=\frac{4\pi \gamma_\alpha^2(1+o(1))}{\log\frac{1}{\mu_\alpha^2}}\,.
\end{split}
\end{equation} 
By \eqref{El1}, \eqref{El4}, and since $\|v_\alpha\|_{H^1_0}^2\le 4\pi+o(1)$, we get that $$\log
\frac{1}{\mu_\alpha^2}\ge (1+o(1))\gamma_\alpha^2\,,
$$ 
which concludes the proof of \eqref{LambdaNotTooSmall} using also \eqref{LambdaTo0} and \eqref{ScalRel}. Now we prove \eqref{MinorVB}. Observe that \eqref{ScalRel}, \eqref{RalphaandMualpha} and \eqref{LambdaNotTooSmall}  imply $r_{\alpha,\delta} \to 0$, as $\alpha\to \lambda_1^-$. Let $\bar{\delta}$ be given in $(0,d(\bar{x},\partial\Omega))$, for $\bar{x}$ as in \eqref{NotToBdry}. Let now $\tilde{A}_\alpha>0$ be such that $\tilde{\Psi}_\alpha:=\tilde{A}_\alpha \log\frac{\bar{\delta}}{|\cdot-x_\alpha|}$ and $\gamma_\alpha-\frac{t_\alpha}{\gamma_\alpha}$ coincide on $\partial B_{x_\alpha}(\mu_\alpha)$. Using \eqref{ScalRel} and \eqref{LambdaNotTooSmall}, we easily get that
\begin{equation}\label{El5}
\tilde{A}_\alpha=\frac{2+o(1)}{\gamma_\alpha}\,.
\end{equation}
But since $0=\Delta \tilde{\Psi}_\alpha\le \Delta v_\alpha $ in $\tilde{\Omega}_\alpha:=B_{x_\alpha}(\bar{\delta})\backslash B_{x_\alpha}(\mu_\alpha)$ and since $\tilde{\Psi}_\alpha\le v_\alpha+o(\gamma_\alpha^{-1})$ in $\partial \tilde{\Omega}_\alpha$, we get from \eqref{FirstResc} and the maximum principle that 
\begin{equation}\label{El6}
\tilde{\Psi}_\alpha\le v_\alpha+o(\gamma_\alpha^{-1}) \text{ in }\tilde{\Omega}_\alpha\,.
\end{equation}
But by \eqref{ScalRel}, \eqref{DefRAlphaDelta}, \eqref{RalphaandMualpha} and  \eqref{El5}, for $z_\alpha\in B_{x_\alpha}(r_{\alpha,\delta})$, we have that 
\[
 \tilde{\Psi}_\alpha(z_\alpha)\ge \frac{1+o(1)}{\gamma_\alpha}\log\frac{1}{r_{\alpha,\delta}^2}\,= \gamma_\alpha (1-\delta+o(1)).
\] 
This concludes the proof of \eqref{MinorVB}, in view of \eqref{FirstResc} and \eqref{El6}.
\end{proof}

\noindent Now, we fix $\delta\in(0,1)$ and we expand $B_\alpha$ up to a distance $r_{\alpha,\delta}$ of $x_\alpha$, as $\alpha\to \lambda_1^-$. As a consequence of Step \ref{Elementary}, we expand $B_\alpha$ up to a distance $r_{\alpha,\delta}$ of $x_\alpha$, as $\alpha\to \lambda_1^-$. Let $S_0$ be the radial solution around $0\in  \mathbb{R}^2$ of 
\begin{equation}\label{EqS}
\Delta S_0-8\exp(-2 T_0)S_0=4\exp(-2T_0)\left(T_0^2-T_0 \right)\,,
\end{equation}
such that $S_0(0)=0$. By \cite{MartMan}, the explicit formula for $S_0$ is
$$S_0(r)=-T_0(r)+\frac{2r^2}{1+r^2}-\frac{1}{2}T_0(r)^2+\frac{1-r^2}{1+r^2}\int_1^{1+r^2}\frac{\log t}{1-t} dt\,, $$
and in particular,
\begin{equation}\label{AsymptExpan}
S_0(r)=\frac{A_0}{4\pi}\log\frac{1}{r^2}+B_0+O\left(\log(r)^2r^{-2} \right)\text{ where }
\begin{cases}A_0=4\pi, \\
 B_0=\frac{\pi^2}{6}+2\,,
\end{cases}
\end{equation}
as $r\to +\infty$. Note that $A_0=\int_{\mathbb{R}^2}(\Delta S_0) dy$. For $0<\lambda_1-\alpha\ll 1$, we let $S_\alpha$ be given by
\begin{equation}\label{DefSAlpha}
S_\alpha(z)=S_0\left(\frac{z-x_\alpha}{\mu_\alpha} \right)\,.
\end{equation}

\begin{Step}\label{StExpansionBubble}
For all sequence $(z_\alpha)_\alpha$ such that $z_\alpha\in B_{x_\alpha}(r_{\alpha,\delta})$, we have that
\begin{equation}\label{EqExpansionBubble}
B_\alpha(z_\alpha)=\gamma_\alpha-\frac{t_\alpha(z_\alpha)}{\gamma_\alpha}+\frac{S_\alpha(z_\alpha)}{\gamma_\alpha^3}+O\left(\frac{1+t_\alpha(z_\alpha)}{\gamma_\alpha^5} \right)\,,
\end{equation}
for all $0<\lambda_1-\alpha\ll 1$.
\end{Step}

As a by-product of Step \ref{StExpansionBubble}, $B_\alpha$ is radially decreasing in $B_{x_\alpha}(r_{\alpha,\delta})$. 
\begin{proof}[Proof of Step \ref{StExpansionBubble}]
Let $w_{1,\alpha}$ be given by 
\begin{equation}\label{Defw1alpha}
B_\alpha=\gamma_\alpha-\frac{t_\alpha}{\gamma_\alpha}+\frac{w_{1,\alpha}}{\gamma_\alpha^3}\,,
\end{equation}
and let $\rho_{1,\alpha}>0$ be defined as
\begin{equation}\label{Rho1Alpha}
\rho_{1,\alpha}=\sup\left\{r\in(0,r_{\alpha,\delta}]\text{ s.t. }\left|S_\alpha-w_{1,\alpha} \right|\le {1+t_\alpha} \text{ in }[0,r] \right\}\,.
\end{equation}
First, we give precise asymptotic expansions of $\Delta w_{1,\alpha}$ in $B_{x_\alpha}(\rho_{1,\alpha})$, as $\alpha\to \lambda_1^-$. We start by proving that the term $A_\alpha B_\alpha$ is well controlled in $B_{x_\alpha}(\rho_{1,\alpha})$, using Step \ref{Elementary}. Indeed, \eqref{ScalRel}, \eqref{DefRAlphaDelta} and \eqref{LambdaNotTooSmall} give
\begin{equation}\label{Neglibible}
\begin{split}
\frac{\exp(t_\alpha(-2+(t_\alpha/\gamma_\alpha^2)))}{\mu_\alpha^2} &=\exp\left(\log\Lambda_\alpha+o(\gamma_\alpha^2) \right)\exp\left(\left(\gamma_\alpha-\frac{t_\alpha}{\gamma_\alpha} \right)^2 \right)\\
&\ge  \exp\left((1-\delta)^2 \gamma_\alpha^2+o(\gamma_\alpha^2)\right)
\end{split}
\end{equation}
in  $B_{x_\alpha}(r_{\alpha,\delta})$. Since $B_\alpha>0$ in $B_{x_\alpha}(\rho_{1,\alpha})$, we get from \eqref{EqB} that $B_\alpha\le \gamma_\alpha$ in this ball. Then  \eqref{Neglibible} implies
\begin{equation}\label{Negligible2}
 A_\alpha B_\alpha \le \lambda_1 \gamma_\alpha=o\left(\frac{\exp(t_\alpha(-2+(t_\alpha/\gamma_\alpha^2)))}{\gamma_\alpha^5\mu_\alpha^2}\right)\text{ in }B_{x_\alpha}(\rho_{1,\alpha})\,.
\end{equation}

Next we observe that  \eqref{AsymptExpan} and \eqref{Rho1Alpha}  imply $ w_{1,\alpha}=O(1+t_\alpha)$ in $B_{x_\alpha}(\rho_{1,\alpha})$. In particular, from \eqref{Defw1alpha} we get
\begin{equation}\label{exp0}
B_\alpha = \gamma_\alpha - \frac{t_\alpha}{\gamma_\alpha} + O\left(\frac{1+t_\alpha}{\gamma_\alpha^3}\right)\,,
\end{equation}
and
\begin{equation}\label{exp02}
B_\alpha^2 = \gamma_\alpha^2 -2t_\alpha + \frac{t_\alpha^2 + 2w_{1,\alpha}}{\gamma_\alpha^2} + O\left(\frac{1+t_\alpha^2}{\gamma_\alpha^4}\right)
\end{equation}
in $B_{x_\alpha}(\rho_{1,\alpha})$. {Since $t_\alpha=O(\gamma_\alpha^2)$ in $B_{x_\alpha}(r_{\alpha,\delta})$}, applying the useful inequality
$${\left|\exp(x)-\sum_{j=0}^{k-1} \frac{x^j}{j!}\right|\le \frac{|x|^k}{k!} \exp(|x|)\,, }$$
for all $x\in\mathbb{R}$ and all integer $k\ge 1$, we obtain that 
\begin{equation}\label{exp1}
\begin{split}
&\exp\left(\frac{t_\alpha^2+2w_{1,\alpha}}{\gamma_\alpha^2}+O\left(\frac{1+t_\alpha^2}{\gamma_\alpha^4} \right)\right)\\ &= 1 + \frac{t_\alpha^2+2w_{1,\alpha}}{\gamma_\alpha^2} + O\left(\frac{(1+t_\alpha^4)\exp(t_\alpha^2 / \gamma_\alpha^2 )}{\gamma_\alpha^4 }\right)\,
\end{split}
\end{equation} 
{in $B_{x_\alpha}(\rho_{1,\alpha})$. Then}, using \eqref{ScalRel}, \eqref{exp0}, \eqref{exp02}, \eqref{exp1}, we get that
\begin{equation}\label{IntermEqExpansionBubble}
\begin{split}
&\Lambda_\alpha B_\alpha\exp(B_\alpha^2)\\
 &=\frac{4\exp(-2t_\alpha)}{\mu_\alpha^2 \gamma_\alpha}\Bigg[1+\frac{2 w_{1,\alpha}+t_\alpha^2-t_\alpha}{\gamma_\alpha^2} +O\left( \frac{(1+t_\alpha^4)\exp(t_\alpha^2/\gamma_\alpha^2)}{\gamma_\alpha^4} \right)\Bigg]
\end{split}
\end{equation}
in $B_{x_\alpha}(\rho_{1,\alpha})$. Now, by \eqref{EqT0}, \eqref{EqB},  \eqref{Negligible2}, and \eqref{IntermEqExpansionBubble}, we get that
\begin{equation}\label{EqW1Alpha}
\Delta w_{1,\alpha}=\frac{4\exp(-2t_\alpha)}{\mu_\alpha^2}\left[2w_{1,\alpha}+t_\alpha^2- t_\alpha+O\left(\frac{(1+t_\alpha^4)\exp(t_\alpha^2/\gamma_\alpha^2)}{\gamma_\alpha^2} \right)\right]
\end{equation}
in $B_{x_\alpha}(\rho_{1,\alpha})$. 

Next, we estimate the growth of the function $w_{1,\alpha}-S_\alpha$. In the sequel, restricting to $B_{x_\alpha}(r_{\alpha,\delta})$ gives that $2-\frac{t_\alpha}{\gamma_\alpha^2}\ge 2-\delta>1$ and then, a sufficiently good decay of the error term $(1+t_\alpha^4)\exp(t_\alpha(-2+(t_\alpha/\gamma_\alpha^2)))$. Namely, we can find $\kappa>1$ and $C>0$ such that 
\begin{equation}\label{Decay}
(1+t_\alpha^4)\exp(t_\alpha(-2+(t_\alpha/\gamma_\alpha^2))) \le C \exp(-\kappa t_\alpha)
\end{equation} 
in $B_{x_\alpha}(r_{\alpha,\delta})$. Now, we observe that 
\begin{equation}\label{FTC}
\int_{B_{x_\alpha}(r)} (\Delta (w_{1,\alpha}-S_\alpha))~ dy=-2\pi r (w_{1,\alpha}-S_\alpha)'(r)\,, 
\end{equation}
and, from \eqref{EqS} and \eqref{EqW1Alpha}, that
\begin{equation}\label{EqTilde}
\Delta\left(w_{1,\alpha}-S_\alpha \right)=\frac{8\exp(-2t_\alpha)}{\mu_\alpha^2} \left[\left(w_{1,\alpha}-S_\alpha \right)+O\left(\frac{(1+t_\alpha^4)\exp(t_\alpha^2/\gamma_\alpha^2)}{\gamma_\alpha^2} \right)\right]\,,
\end{equation}
for all $0\le r\le \rho_{1,\alpha}$. By \eqref{Decay}, we get that 
\begin{equation}\label{Int1}
\int_{B_{x_\alpha}(r)} \frac{8\exp(-2t_\alpha+t_\alpha^2/\gamma_\alpha^2)(1+t_\alpha^4)}{\mu_\alpha^2}  dy \le \frac{8\pi}{\kappa-1}\left(1-(1+(r/\mu_\alpha)^2)^{1-\kappa}\right)\,,
\end{equation}
and, since $|(w_{1,\alpha}-S_\alpha)(r)|\le \|(w_{1,\alpha}-S_\alpha)'\|_{L^\infty([0,\rho_{1,\alpha}])}r$,  that 
\begin{equation}\label{Int2}
\int_{B_{x_\alpha}(r)} \frac{8\exp(-2t_\alpha)}{\mu_\alpha^2} \left|w_{1,\alpha}-S_\alpha \right| dy \le \mu_\alpha h(r/\mu_\alpha) \|(w_{1,\alpha}-S_\alpha)'\|_{L^\infty([0,\rho_{1,\alpha}])}\,,
\end{equation}
where
\[
h(s)= 8\pi\left( \arctan s- \frac{s}{1+s^2}\right)\,, \quad s\ge0\,.
\]
Then, by \eqref{FTC}, \eqref{EqTilde}, \eqref{Int1}, and \eqref{Int2}, there exists a constant $C'>{1}$ such that
\begin{equation}\label{CrucialGradEst1}
\begin{split}
\frac{r|(w_{1,\alpha}-S_\alpha)'(r)|}{C'}\le& \frac{(r/\mu_\alpha)^2}{\gamma_\alpha^2\left(1+(r/\mu_\alpha)^2\right)}\\
&~~+\frac{ \mu_\alpha\|(w_{1,\alpha}-S_\alpha)'\|_{L^\infty([0,\rho_{1,\alpha}])}(r/\mu_\alpha)^3}{1+(r/\mu_\alpha)^3}\,
\end{split}
\end{equation} 
for all $0\le r\le \rho_{1,\alpha}$ and all $0<\lambda_1-\alpha\ll 1$. Now we prove that 
\begin{equation}\label{ToProve1}
\mu_\alpha \|(w_{1,\alpha}-S_\alpha)'\|_{L^\infty([0,\rho_{1,\alpha}])}=O\left(\frac{1}{\gamma_\alpha^2}\right)\,.
\end{equation}
Otherwise, we assume by contradiction that
\begin{equation}\label{ToProve2}
\gamma_\alpha^2\mu_\alpha \|(w_{1,\alpha}-S_\alpha)'\|_{L^\infty([0,\rho_{1,\alpha}])}=\gamma_\alpha^2 \mu_\alpha |(w_{1,\alpha}-S_\alpha)'|(s_\alpha)\to +\infty
\end{equation}
as $\alpha\to \lambda_1^-$, for $s_\alpha\in (0,\rho_{1,\alpha}]$. Up to a subsequence, we may assume that
\begin{equation}\label{ToProve3}
\frac{\rho_{1,\alpha}}{\mu_\alpha}\to \delta_0
\end{equation}
as $\alpha\to \lambda_1^-$, for some $\delta_0\in (0,+\infty]$. Note that \eqref{CrucialGradEst1} with \eqref{ToProve2} gives $s_\alpha=O(\mu_\alpha)$, $\mu_\alpha=O(s_\alpha)$ and then $\delta_0>0$. Let $\tilde{w}_\alpha$ be given by
\begin{equation*}
\tilde{w}_\alpha(s)=\frac{(w_{1,\alpha}-S_\alpha)(\mu_\alpha s)}{\mu_\alpha \|(w_{1,\alpha}-S_\alpha)'\|_{L^\infty([0,\rho_{1,\alpha}])}}\,,
\end{equation*}
so that, by  \eqref{CrucialGradEst1} and \eqref{ToProve2}, there exists a constant $C''>0$ such that
\begin{equation}\label{ToProve4}
|\tilde{w}_\alpha'(s)|\le \frac{C''}{1+s}\text{ in }[0,\rho_{1,\alpha}/\mu_\alpha]\,,
\end{equation}
for all $0<\lambda_1-\alpha\ll 1$. Then, by \eqref{EqTilde}, \eqref{ToProve4} and elliptic theory, we get that there exists $\tilde{w}$ such that
\begin{equation}\label{ToProve5}
\begin{split}
&\tilde{w}_\alpha\to \tilde{w}\text{ in }C^1_{loc}(B_0(\delta_0))\text{ as }\alpha\to \lambda_1^-\,,
\end{split}
\end{equation}
and $\tilde{w}$ solves
\begin{equation}\label{ToProve5Bis}
\begin{cases}
&\Delta \tilde{w}=8\exp(-2T_0) \tilde{w}\text{ in }B_0(\delta_0)\,,\\
&\tilde{w}(0)=0\,,\\
&\tilde{w} \text{ radially symmetric around }0\in\mathbb{R}^2\,;
\end{cases}
\end{equation}
but \eqref{ToProve5Bis} implies 
\begin{equation}\label{ToProve6}
\tilde{w}\equiv 0\text{ in }B_0(\delta_0)\,.
\end{equation}
By \eqref{ToProve4}, \eqref{ToProve5}, \eqref{ToProve6} and the dominated convergence theorem we get
\begin{equation}\label{ToProve7}
\int_{B_{x_\alpha}(\rho_{1,\alpha})} \frac{\exp(-2t_\alpha)}{\mu_\alpha^2} |w_{1,\alpha}-S_\alpha| dy=o(\mu_\alpha \|(w_{1,\alpha}-S_\alpha)'\|_{L^\infty([0,\rho_{1,\alpha}])})\,.
\end{equation}
Resuming now the argument to get \eqref{CrucialGradEst1}, but replacing  \eqref{Int2} with \eqref{ToProve7}, and using  \eqref{ToProve2},  we get 
\begin{equation}\label{CrucialGradEst2}
\begin{split}
{r|(w_{1,\alpha}-S_\alpha)'(r)|}=o\left({ \mu_\alpha\|(w_{1,\alpha}-S_\alpha)'\|_{L^\infty([0,\rho_{1,\alpha}])}}\right)\,
\end{split}
\end{equation}
for all $0\le r\le \rho_{1,\alpha}$ and as $\alpha\to \lambda_1^-$. But \eqref{CrucialGradEst2} is clearly not possible at $s_\alpha$. This concludes the proof of \eqref{ToProve1}. 

Now, plugging \eqref{ToProve1} in \eqref{CrucialGradEst1}, using that $w_{1,\alpha}(0)=S_\alpha(0)=0$ and the fundamental theorem of calculus, we get that 
$$\left\|w_{1,\alpha}-S_\alpha\right\|_{L^\infty([0,\rho_{1,\alpha}])}=O\left(\frac{1+t_\alpha}{\gamma_\alpha^2} \right) $$
as $\alpha\to \lambda_1^-$, which, in view of \eqref{Rho1Alpha}, gives $\rho_{1,\alpha}=r_{\alpha,\delta}$ and concludes the proof of Step \ref{StExpansionBubble}.
\end{proof}

Now, we compare the behavior of $v_\alpha$ and $B_\alpha$ in $B_{x_\alpha}(r_{\alpha,\delta})$. Let $\kappa$ be any fixed number in $(0,1)$. Let $r_\alpha$ be given by
\begin{equation}\label{RAlphaDT}
r_\alpha=\sup\left\{r\in(0,r_{\alpha,\delta}]\text{ s.t. }|\bar{v}_\alpha-B_\alpha|\le \frac{\kappa}{\gamma_\alpha}\text{ in }B_{x_\alpha}(r) \right\}\,.
\end{equation}
We get from \eqref{WeakGradEst} and \eqref{MinorVB} that
\begin{equation}\label{NablaV}
|\cdot-x_\alpha||\nabla v_\alpha|\le \frac{D_0}{(1-\delta+o(1))\gamma_\alpha }\text{ in }B_{x_\alpha}(r_{\alpha,\delta})\,.
\end{equation}
Then letting $w_\alpha$ be given by 
\begin{equation}\label{DefW}
v_\alpha=B_\alpha+w_\alpha\,,
\end{equation}
we get from \eqref{RAlphaDT} and \eqref{NablaV} that
\begin{equation}\label{ContW}
|w_\alpha|\le \left(\kappa+\frac{D_0\pi}{(1-\delta+o(1))} \right)\frac{1}{\gamma_\alpha}\quad\text{ in }B_{x_\alpha}(r_\alpha)\,.
\end{equation}
Then, we obtain from \eqref{CBis}, \eqref{EqB} and \eqref{ContW} that there exists a constant $D_1>0$ such that
\begin{equation}\label{Lin2}
\begin{split}
|\Delta w_\alpha| & \le \left(\lambda_1 +D_1\left(1+2B_\alpha^2 \right)\exp(B_\alpha^2)\right)|w_\alpha|\,,\\
& \le D_1 \left(1+2B_\alpha^2 \right)\exp(B_\alpha^2) (1+o(1))|w_\alpha|
\end{split}
\end{equation}
in $B_{x_{\alpha}}(r_\alpha)$, using also {\eqref{DefRAlphaDelta}}, \eqref{LambdaNotTooSmall}, {\eqref{AsymptExpan}} and \eqref{EqExpansionBubble} to get $\exp(B_\alpha^2)\gg \lambda_1$. Summarizing, the $v_\alpha$'s satisfy \eqref{CBis} and \eqref{NablaV}, and the $B_\alpha$'s satisfy \eqref{EqExpansionBubble} in $B_{x_\alpha}(r_{\alpha,\delta})$, while \eqref{AsymptExpan} holds true. Moreover, the $w_\alpha$'s satisfy \eqref{Lin2} in $B_{x_{\alpha}}(r_\alpha)$. Then, arguing exactly as in \cite[Section 3]{DruThiI} dealing with the case $A_\alpha=0$, we get the following result.

\begin{Step}\label{StComparisonVB}
Let $\delta\in(0,1)$ be given. Then we have that $r_\alpha=r_{\alpha,\delta}$ and, in other words,
$$\left\|\bar{v}_\alpha-B_\alpha \right\|_{L^\infty\left(B_{x_\alpha}(r_{\alpha,\delta}) \right)}=o\left(\frac{1}{\gamma_\alpha}\right)\,, $$
as $\alpha\to \lambda_1^-$. Moreover, we have that
\begin{equation*}
\left\|\nabla(v_\alpha-B_\alpha) \right\|_{L^\infty(B_{x_\alpha}(r_{\alpha,\delta}))}=O\left(\frac{1}{\gamma_\alpha r_{\alpha,\delta}} \right)
\end{equation*}
and then, there exists a constant $C>0$ such that
\begin{equation}\label{StrongDiff}
|v_\alpha-B_\alpha|\le C\frac{|\cdot-x_\alpha|}{\gamma_\alpha r_{\alpha,\delta}}\text{ in }B_{x_\alpha}(r_{\alpha,\delta})\,.
\end{equation}
\end{Step}

As a direct consequence of Steps \ref{StExpansionBubble} and  \ref{StComparisonVB} we get the asymptotic expansion 
\[
\begin{split}
f_\alpha(v_\alpha) &= O(\gamma_\alpha) + \Lambda_\alpha \left(B_\alpha+ O\left( \frac{|\cdot-x_\alpha|}{r_{\alpha,\delta}\gamma_\alpha}\right)\right)\exp\left(B_\alpha^2 + O\left(\frac{|\cdot-x_\alpha|}{r_{\alpha,\delta}}\right)\right)\\
&= O(\gamma_\alpha) + \Lambda_\alpha B_\alpha \exp(B_\alpha^2) \left(1+O\left( \frac{|\cdot-x_\alpha|}{r_{\alpha,\delta}}\right) \right).
\end{split}
\]
Then, expanding as in $\eqref{Negligible2}$ and $\eqref{IntermEqExpansionBubble}$, we find that 
\begin{equation}
\begin{split}
f_\alpha(v_\alpha) &=\frac{4\exp(-2t_\alpha)}{\mu_\alpha^2 \gamma_\alpha}\Bigg[1+\frac{2 S_\alpha+t_\alpha^2-t_\alpha}{\gamma_\alpha^2} \\ 
& \quad +O\left( \frac{(1+t_\alpha^4)}{\gamma_\alpha^4} \exp\left(\frac{t_\alpha^2}{\gamma_\alpha^2}\right) +{\left(1+\frac{t_\alpha^2}{\gamma_\alpha^2}\right)}\frac{|\cdot-x_\alpha|}{r_{\alpha,\delta}}\right)  \Bigg]
\end{split}
\end{equation}
in $B_{x_\alpha}(r_{\alpha,\delta})$. Since $\delta<1$, we can argue as in \eqref{Decay} to estimate the exponential in the error term. Specifically, we can find $\kappa>1$ such that 
\begin{equation}\label{NewExpfalpha}
\begin{split}
f_\alpha(v_\alpha) =\frac{4\exp(-2t_\alpha)}{\mu_\alpha^2 \gamma_\alpha}\Bigg[1&+\frac{2 S_\alpha+t_\alpha^2-t_\alpha}{\gamma_\alpha^2} +O\left({\left(1+\frac{t_\alpha^2}{\gamma_\alpha^2}\right)}\frac{|\cdot-x_\alpha|}{r_{\alpha,\delta}}\right)  \Bigg] \\ &+ O\left(\frac{\exp\left(-\kappa t_\alpha\right)}{\mu_\alpha^2 \gamma_\alpha^4}\right).
\end{split}
\end{equation} 
Similarly, we obtain
\begin{equation}\label{NewExpEnergy}
\begin{split}
v_\alpha f_\alpha(v_\alpha) =\frac{4\exp(-2t_\alpha)}{\mu_\alpha^2}\Bigg[1&+\frac{2 S_\alpha +t_\alpha^2-2t_\alpha}{\gamma_\alpha^2}  + O\left({\left(1+\frac{t_\alpha^2}{\gamma_\alpha^2}\right)}\frac{|\cdot-x_\alpha|}{r_{\alpha,\delta}}\right) \Bigg] \\  &+ O\left(\frac{\exp\left(-\kappa t_\alpha\right)}{\mu_\alpha^2 \gamma_\alpha^4}\right).
\end{split}
\end{equation}
in $B_{x_\alpha}(r_{\alpha,\delta})$. 

Now we focus on the behavior of the $v_\alpha$'s in $\Omega\backslash B_{x_\alpha}(r_{\alpha,\delta})$. Assume that $0<\delta'<\delta<1$. We let $\tilde{v}_\alpha$ be given by
\begin{equation}\label{DefVTilde}
\tilde{v}_\alpha=
\begin{cases}
&v_\alpha\text{ in }\Omega\backslash B_{x_\alpha}(r_{\alpha,\delta})\,,\\
&\min\left(v_\alpha,(1-\delta')\gamma_\alpha \right)\text{ in }B_{x_\alpha}(r_{\alpha,\delta})\,.
\end{cases}
\end{equation}
Note that 
\begin{equation}\label{PropRk}
v_\alpha<(1-\delta')\gamma_\alpha \text{ in }\partial B_{x_\alpha}(r_{\alpha,\delta})
\end{equation}
 by \eqref{EqExpansionBubble} and \eqref{StrongDiff}. Then we have that $\tilde{v}_\alpha\in H^1_0$ and that $v_\alpha=\tilde{v}_\alpha+\tilde{v}_{1,\alpha}$, where $\tilde{v}_{1,\alpha}:=\mathds{1}_{B_{x_\alpha}(r_{\alpha,\delta})}(v_\alpha-(1-\delta')\gamma_\alpha)^+$ and $t^+=\max(t,0)$. Now, by \eqref{PropRk} and continuity, we have that $\tilde{v}_{1,\alpha}$ is zero in a neighborhood of $\partial B_{x_\alpha}(r_{\alpha,\delta})$. Then, for any given $R>0$, we can compute 
 \begin{equation*}
 \begin{split}
 \int_{B_{x_\alpha}(r_{\alpha,\delta})}|\nabla \tilde{v}_{1,\alpha}|^2 dy & =\int_{B_{x_\alpha}(r_{\alpha,\delta})}\nabla \tilde{v}_{1,\alpha}(y).\nabla v_\alpha dy\\
 & =\int_{B_{x_\alpha}(r_{\alpha,\delta})} (\Delta v_\alpha) \tilde{v}_{1,\alpha} dy\\
 & \ge \int_{B_{x_\alpha}(R\mu_\alpha)} f_\alpha(v_\alpha) \tilde{v}_{1,\alpha} dy\\
 & \ge \delta' (1+o(1))\int_{B_0(R)}  \frac{4}{(1+|z|^2)^2}  dz
 \end{split}
 \end{equation*}
for $0<\lambda_1-\alpha\ll 1$, since $r_{\alpha,\delta}/\mu_\alpha\to +\infty$ and using \eqref{ScalRel} and \eqref{FirstResc}. {Since $R>0$ is arbitrary, we obtain 
\[
\|\tilde v_{1,\alpha}\|_{H^1_0}\ge 4\pi \delta' (1+o(1)). 
\]
Since \eqref{St1} implies $\|\tilde{v}_\alpha+\tilde{v}_{1,\alpha}\|^2_{H^1_0}=4\pi+o(1)$, and since $\tilde{v}_{1,\alpha}$ and $\tilde{v}_\alpha$ are $H^1_0$-orthogonal, we get that} 
\begin{equation} \label{BdTechnical2}
\|\tilde{v}_\alpha\|_{H^1_0}^2\le 4\pi(1-\delta'+o(1))\,.
\end{equation}
 Moreover, since $\delta$ and $\delta'$ may be arbitrarily close to $1$ in the above argument, we can check that, up to a subsequence, $v_\alpha\rightharpoonup 0$ weakly in $H^1_0$ and then that
\begin{equation}\label{VTo0}
v_\alpha\to 0\text{ strongly in }L^p\,,
\end{equation}
{for any $p\ge 1$}, as $\alpha\to \lambda_1^-$. Furthermore, by \eqref{BdTechnical2} and Moser's inequality, there exists  $p'>1$ such that $(\exp(\tilde{v}_\alpha^2))_\alpha$ is bounded in $L^{p'}(\Omega)$. Using \eqref{EqExpansionBubble} and \eqref{StrongDiff}, we can also check that $\tilde{v}_\alpha=v_\alpha$ in $\Omega\backslash B_{x_\alpha}(r_{\alpha,\delta}/2)${. Then, we get that}
\begin{equation}\label{BdTechnical}
(\exp(v_\alpha^2))_\alpha \text{ is bounded in }L^{p'}(\Omega\backslash B_{x_\alpha}(r_{\alpha,\delta}/2))\,.
\end{equation}
{From now on,  we fix $p\ge 2$ }and $r>1$ such that
\begin{equation}\label{PrepaHolder}
\frac{1}{p'}+\frac{1}{p}+\frac{1}{r}=1\,.
\end{equation}
In the sequel, $v$ is the unique function characterized by
\begin{equation}\label{DefV}
\begin{cases}
&\Delta v=\lambda_1 v\,,\quad v>0\text{ in }\Omega\,,\\
&v=0\text{ in }\partial\Omega\,,\\
&\|v\|_{p}=1\,.
\end{cases}
\end{equation}

\begin{Step}\label{StGreenFun}
For all sequence $(z_\alpha)_\alpha$ of points such that $z_\alpha\in \Omega\backslash B_{x_\alpha}(r_{\alpha,\delta})$, we have that
\begin{equation}\label{C0Theory}
v_\alpha(z_\alpha)=\|v_\alpha\|_{p}v(z_\alpha)+o\left(\|v_\alpha\|_p \right)+\frac{1}{\gamma_\alpha}\log \frac{1}{|x_\alpha-z_\alpha|^2}+O\left(\frac{1}{\gamma_\alpha}\right)\,,
\end{equation}
for all $0<\lambda_1-\alpha\ll 1$, where $p$ {is} as in \eqref{PrepaHolder} and $v$ is as in \eqref{DefV}. Moreover, \eqref{F0} and \eqref{AccurateLambdaEstim} hold true.
\end{Step}

\begin{proof}[Proof of Step \ref{StGreenFun}] Let $(z_\alpha)_\alpha$ be a sequence of points such that $z_\alpha\in \Omega\backslash B_{x_\alpha}(r_{\alpha,\delta})$ for all $\alpha$. Let $G$ be the Green's function of $\Delta$ in $\Omega$ with zero Dirichlet boundary conditions. Then (see for instance \cite[Appendix B]{DruThiI}), there exists a constant $C>0$ such that
\begin{equation}\label{EstimGreenFun}
\begin{split}
&0<G_x(y)\le \frac{1}{2\pi} \log \frac{C}{|x-y|}\\
&|\nabla G_x(y)|\le \frac{C}{|x-y|}
\end{split}
\end{equation}
for all $x\neq y$ in $\Omega$. By the Green's representation formula and \eqref{CBis}, we get that
\begin{equation}\label{GreenFun1}
v_\alpha(z_\alpha)=\int_\Omega G_{z_\alpha}(y) f_\alpha(v_\alpha(y)) dy\,.
\end{equation}
Now, we split the integral in \eqref{GreenFun1} according to $\Omega=B_{x_\alpha}\left(\frac{r_{\alpha,\delta}}{2} \right)\cup B_{x_\alpha}\left(\frac{r_{\alpha,\delta}}{2}\right)^c$, where $B_{x_\alpha}\left(\frac{r_{\alpha,\delta}}{2}\right)^c=\Omega\backslash B_{x_\alpha}\left(\frac{r_{\alpha,\delta}}{2}\right)$. {First, integrating \eqref{NewExpfalpha} and using the dominated convergence theorem, we get that
\begin{equation}\label{IntExpansionLocal}
\begin{split}
\int_{B_{x_\alpha}\left(\frac{r_{\alpha,\delta}}{2}\right)}f_\alpha(v_\alpha) dy   & = \int_{B_{x_\alpha}\left(\frac{r_{\alpha,\delta}}{2}\right) } \frac{4\exp\left(-2 t_\alpha\right)}{\mu_\alpha^2\gamma_\alpha} dy \\ & \quad + O\left(\frac{1}{\gamma_\alpha^3}\right)  + O\left( \frac{\mu_\alpha}{r_{\alpha,\delta}\gamma_\alpha}\right). 
\end{split}
\end{equation}
By  \eqref{RalphaandMualpha} we know that  
\begin{equation}\label{ExponetialRatio}
\frac{r_{\alpha,\delta}^2}{\mu_\alpha^2} = \exp\left( \delta \gamma_\alpha^2  + o(1)\right)\,,
\end{equation} 
as $\alpha\to \lambda_1^-$. Then, from \eqref{IntExpansionLocal} and \eqref{ExponetialRatio} we get  that 
\begin{equation}\label{ExpansionLocal}
\begin{split}
\int_{B_{x_\alpha}\left(\frac{r_{\alpha,\delta}}{2}\right)}f_\alpha(v_\alpha) dy   & = \frac{4\pi}{\gamma_\alpha} + O\left(\frac{1}{\gamma_\alpha^3}\right),
\end{split}
\end{equation}
for all $0<\lambda_1-\alpha\ll 1$.} Independently, by \eqref{EstimGreenFun}, we get that there exists $C>0$ such that
\begin{equation}\label{MeanValue}
\left|G_{z_\alpha}(y)-G_{z_\alpha}(x_\alpha) \right|\le \frac{C |y-x_\alpha|}{r_{\alpha,\delta}}
\end{equation}
for all $y\in B_{x_\alpha}\left(\frac{r_{\alpha,\delta}}{2}\right)$ and all $\alpha$. {Then, from \eqref{NewExpfalpha}, \eqref{ExpansionLocal} and \eqref{MeanValue} we obtain that
\begin{equation}\label{Inter0}
\begin{split}
&\int_{B_{x_\alpha}\left(\frac{r_{\alpha,\delta}}{2} \right)} G_{z_\alpha}(y) f_{\alpha}(v_\alpha(y))dy\\ &  \qquad= \left( \frac{4\pi}{\gamma_\alpha} + O\left(\frac{1}{\gamma_\alpha^3}\right)\right) G_{z_\alpha}(x_\alpha)  + \int_{B_{x_\alpha}\left(\frac{r_{\alpha,\delta}}{2}\right)}f_\alpha(v_\alpha(y)) |y-x_\alpha|  dy\, {.}
\end{split}
\end{equation}
But \eqref{NewExpfalpha}, the dominated convergence theorem and \eqref{ExponetialRatio} give that 
\begin{equation}\label{Inter1}
\begin{split}
\int_{B_{x_\alpha}\left(\frac{r_{\alpha,\delta}}{2}\right)}f_\alpha(v_\alpha(y)) |y-x_\alpha|  dy &= O\left(\int_{B_{x_\alpha}\left(\frac{r_{\alpha,\delta}}{2}\right)}\frac{\exp\left(-\kappa t_\alpha \right)|y-x_\alpha|}{\gamma_\alpha \mu_\alpha^2 r_{\alpha,\delta}} dy \right)\\
& = o\left(\frac{1}{\gamma_\alpha}\right).
\end{split}
\end{equation}
Then, from \eqref{Inter0} and \eqref{Inter1},} we get that 
\begin{equation}\label{EstimInsideBall}
\int_{B_{x_\alpha}\left(\frac{r_{\alpha,\delta}}{2} \right)} G_{z_\alpha}(y) f_{\alpha}(v_\alpha(y))dy=\left[\frac{4\pi}{\gamma_\alpha}+O\left(\frac{1}{\gamma_\alpha^3} \right)\right] G_{z_\alpha}(x_\alpha)+o\left(\frac{1}{\gamma_\alpha} \right)\,.
\end{equation}
Now we turn to the integral in $B_{x_\alpha}\left(\frac{r_{\alpha,\delta}}{2} \right)^c$. {By Hölder's inequality, \eqref{LambdaTo0}, \eqref{BdTechnical}, \eqref{PrepaHolder} and \eqref{EstimGreenFun}, t}here exists $C>0$ such that
\begin{equation}\label{Inter2}
\begin{split}
&\int_{B_{x_\alpha}\left(\frac{r_{\alpha,\delta}}{2} \right)^c} G_{z_\alpha}(y) f_\alpha(v_\alpha(y)) dy\\ 
&\le C \|G_{z_\alpha}\|_{L^r} \|v_\alpha\|_{L^p}\left(\lambda_1+ \Lambda_\alpha \|\exp(v_\alpha^2)\|_{L^{p'}\left(B_{x_\alpha}\left(\frac{r_{\alpha,\delta}}{2} \right)^c \right)}\right)=O\left( \|v_\alpha\|_{L^p}\right)
\end{split}
\end{equation}
for all $\alpha$. Putting together \eqref{EstimGreenFun}, {\eqref{GreenFun1}}, \eqref{EstimInsideBall} and \eqref{Inter2}, we have obtained that there exists $C,\bar{C}>0$ such that
\begin{equation}\label{AFirstPointwiseEstim}
v_\alpha(z_\alpha)\le (1+o(1))\frac{\log\frac{C}{|x_\alpha-z_\alpha|^2}}{\gamma_\alpha} +\bar{C}\|v_\alpha\|_{L^p} \,,
\end{equation}
as $\alpha\to \lambda_1^-$. Now we prove \eqref{F0}, which implies
\begin{equation}\label{VpLarge}
\gamma_\alpha \|v_\alpha\|_p\to +\infty\,,
\end{equation}
as $\alpha\to \lambda_1^-$, since {$p\ge 2$}. We multiply \eqref{CBis} by $v$ as in \eqref{DefV}  and integrate in $\Omega$. We get 
\begin{equation}\label{Inter3}
\left(\lambda_1-A_\alpha \right)\int_\Omega v v_\alpha dy=\Lambda_\alpha \int_\Omega v v_\alpha \exp(v_\alpha^2) dy\ge \frac{4\pi v(\bar{x})(1+o(1))}{\gamma_\alpha}
\end{equation}
as $\alpha\to \lambda_1^-$, where $\bar{x}$ is as in \eqref{NotToBdry}, using \eqref{FirstResc}. Since $A_\alpha\to \lambda_1^-$ by \eqref{St1}, we get \eqref{F0} and \eqref{VpLarge} from \eqref{Inter3} and the Cauchy-Schwarz inequality. Now we prove that 
\begin{equation}\label{ConvToV}
\frac{v_\alpha}{\|v_\alpha\|_p}\to v\text{ in }C^1_{loc}(\bar{\Omega}\backslash\{\bar{x}\})\,,
\end{equation}
as $\alpha\to \lambda_1^-$. By \eqref{LambdaTo0}, \eqref{CBis}, \eqref{VTo0}, \eqref{AFirstPointwiseEstim}, \eqref{VpLarge} and elliptic theory, we get that $({v_\alpha}/{\|v_\alpha\|_p})_\alpha$ converges in $C^1_{loc}(\bar{\Omega}\backslash\{\bar{x}\})$ to some $\tilde{v}$ solving
\begin{equation}\label{Inter4}
\Delta \tilde{v}=\lambda_1 \tilde{v}\,, 
\end{equation}
in $\Omega\backslash \{\bar{x}\}$. But, by \eqref{AFirstPointwiseEstim} and \eqref{VpLarge} again, we get that $0\le \tilde{v}\le \bar{C}$ in $\Omega\backslash \{\bar{x}\}$ for $\bar{C}$ as in \eqref{AFirstPointwiseEstim}, that $\tilde{v}$ solves \eqref{Inter4} in $\Omega$ and that $\|\tilde{v}\|_p=1$. Then $\tilde{v}=v$ and \eqref{ConvToV} is proved. Since $r_{\alpha,\delta}\to 0$ and by  \eqref{ConvToV}, we can find a sequence $(\delta_\alpha)_\alpha$ of positive real numbers converging to $0$ such that $\delta_\alpha\ge r_{\alpha,\delta}/2$ and such that
\begin{equation}\label{ConvToV2}
\left\|\frac{v_\alpha}{\|v_\alpha\|_p}-v \right\|_{C^0(\Omega\backslash B_{x_\alpha}(\delta_\alpha))}=o(1)
\end{equation}
as $\alpha\to \lambda_1^-$. Note that \eqref{C0Theory} is already proved by \eqref{NotToBdry}, \eqref{VpLarge} and \eqref{ConvToV2}, if
\begin{equation}\label{AssumptC0Thy}
\liminf_{\alpha\to \lambda_1} |\bar{x}-z_\alpha|>0\,.
\end{equation}
Then, in order to conclude the proof of \eqref{C0Theory}, we assume now that
\begin{equation}\label{AssumptC0Thy2}
z_\alpha\to \bar{x}
\end{equation}
as $\alpha\to\lambda_1^-$. Then, it is known that
\begin{equation}\label{EstimGreenFun2}
G_{z_\alpha}(y)=\frac{1}{4\pi}\log \frac{1}{|y-z_\alpha|^2}+O(1)
\end{equation}
for all $y\neq z_\alpha$ and all $\alpha$. Using {\eqref{ScalRel} and \eqref{RalphaandMualpha}}, we get that
\begin{equation}\label{EstimateRAlphaDelta}
\log \frac{1}{r_{\alpha,\delta}^2}=(1-\delta+o(1))\gamma_\alpha^2\,.
\end{equation}
 Since $z_\alpha\in B_{x_\alpha}\left({r_{\alpha,\delta}} \right)^c$, we get from \eqref{EstimInsideBall}, \eqref{EstimGreenFun2}, \eqref{EstimateRAlphaDelta} that
 \begin{equation}\label{EstimInsideBall2}
 \int_{B_{x_\alpha}\left(\frac{r_{\alpha,\delta}}{2} \right)} G_{z_\alpha}(y) f_{\alpha}(v_\alpha(y))dy=\frac{1}{\gamma_\alpha}\log \frac{1}{|x_\alpha-z_\alpha|^2}+O\left(\frac{1}{\gamma_\alpha} \right)\,.
 \end{equation}
 Now, since $\Delta v=\lambda_1 v$ and $A_\alpha\to \lambda_1^-$, we get by \eqref{LambdaTo0}, \eqref{VTo0}, \eqref{EstimGreenFun} and \eqref{ConvToV2} that
\begin{equation}\label{EstimFarFromXAlpha}
\begin{split} 
\int_{B_{x_\alpha}\left(\delta_\alpha\right)^c} G_{z_\alpha}(y) f_{\alpha}(v_\alpha(y))dy & { =  \|v_\alpha\|_p\lambda_1 \int_{\Omega} G_{z_\alpha }(y) v(y) dy + o(\|v_\alpha\|_p) }\\ 
&= \|v_\alpha\|_p v(z_\alpha)+o(\|v_\alpha\|_p)\,,
\end{split}
\end{equation}
 as $\alpha\to \lambda_1^-$. 
{Now we denote $\Omega_\alpha=B_{x_\alpha}(\delta_\alpha)\backslash B_{x_\alpha}\left(\frac{r_{\alpha,\delta}}{2} \right)$}. {On the one hand,  using \eqref{EstimGreenFun}, \eqref{AFirstPointwiseEstim}, \eqref{VpLarge} and $\delta_\alpha\to 0$ as $\alpha\to \lambda_1^-$,} we get that
 \begin{equation}\label{Inter5}
 \int_{\Omega_\alpha} G_{z_\alpha}(y) A_\alpha v_\alpha dy=O\left(\delta_\alpha^2 \log\frac{1}{\delta_\alpha}\|v_\alpha\|_p\right)+O\left(\frac{1}{\gamma_\alpha}\right)=o(\|v_\alpha\|_p).
\end{equation}
On the other hand, {using \eqref{LambdaTo0}, \eqref{GreenFun1}  \eqref{AFirstPointwiseEstim}, and the dominated convergence theorem} we have that
 \begin{equation}\label{Inter9}
 \begin{split}
 &\int_{\Omega_\alpha} G_{z_\alpha}(y) \Lambda_\alpha v_\alpha\exp(v_\alpha^2) dy\\
 &=o\Bigg( \int_{\Omega_\alpha} \log\frac{C}{|z_\alpha-y|}\left(\frac{1}{\gamma_\alpha}\log\frac{C}{|x_\alpha-y|}+\|v_\alpha\|_p\right)\times\\
 &\quad\quad\quad\quad\quad\quad\exp\left(\left[\frac{1+o(1)}{\gamma_\alpha}\log\frac{1}{|x_\alpha-y|^2}+o(1)\right]^2 \right) dy \Bigg)\\
& = o\left(\|v_\alpha\|_p\right)\,{.} \end{split}
 \end{equation}
Indeed, \eqref{EstimateRAlphaDelta} gives that
\begin{equation}\label{Inter10}
\frac{1}{\gamma_\alpha^2}\log\frac{1}{|x_\alpha-\cdot|^2}\le 1-\delta+o(1)<1
\end{equation}
in $B_{x_{\alpha}}\left(r_{\alpha,\delta}/2\right)^c$, for all $0<\lambda_1-\alpha\ll 1$. Combining \eqref{EstimInsideBall2}, \eqref{EstimFarFromXAlpha}, \eqref{Inter5} and \eqref{Inter9}, we get \eqref{C0Theory}. At last we prove \eqref{AccurateLambdaEstim}.
 By \eqref{ScalRel}, \eqref{EqExpansionBubble}, \eqref{StrongDiff}, we have that
 \begin{equation}\label{Inter7}
 v_\alpha(\tilde{z}_\alpha)=B_\alpha(\tilde{z}_\alpha)+O\left(\frac{1}{\gamma_\alpha} \right)=\frac{1}{\gamma_\alpha}\left(\log \frac{1}{|x_\alpha-\tilde{z}_\alpha|^2}+\log \frac{1}{\gamma_\alpha^2\Lambda_\alpha}+O(1) \right)
 \end{equation}
for all $\alpha$, where $(\tilde{z}_\alpha)_\alpha$ is given such that $\tilde{z}_\alpha\in\partial B_{x_\alpha}(r_{\alpha,\delta})$. But picking $z_\alpha=\tilde{z}_\alpha$ in {\eqref{C0Theory}}, we get from \eqref{VpLarge}, \eqref{AssumptC0Thy2} and \eqref{Inter7} that
\begin{equation}\label{Inter8}
\log\frac{1}{\gamma_\alpha^2\Lambda_\alpha} {=} \gamma_\alpha\|v_\alpha\|_p v(\bar{x})(1+o(1))\to +\infty\,,
\end{equation}
as $\alpha\to \lambda_1^-$, which concludes the proof of \eqref{AccurateLambdaEstim} and that of Step \ref{StGreenFun}.
\end{proof}

\noindent In order to conclude the proof of Proposition \ref{MainProp}, it remains to prove \eqref{F}. 
\begin{proof}[Proof of Proposition \ref{MainProp} (ended)] {By \eqref{CBis}}, in order to get \eqref{F}, it is sufficient to prove that
\begin{equation}\label{FToProve}
\Lambda_\alpha \int_\Omega v_\alpha^2 \exp(v_\alpha^2) dy={4\pi+}o\left(\left( \int_\Omega v_\alpha^2 dy\right)^2 \right){.}
\end{equation}
First{, using \eqref{NewExpEnergy} and the dominated convergence theorem,}  we get that 
\begin{equation}\label{IntFToProve3}
\begin{split}
&\Lambda_\alpha \int_{B_{x_\alpha}(r_{\alpha,\delta})} v_\alpha^2 \exp(v_\alpha^2) dy \\
&=\int_{B_{x_\alpha}(r_{\alpha,\delta})}  \frac{4\exp(-2t_\alpha)}{\mu_\alpha^2}\Bigg(1+\frac{2S_\alpha+t_\alpha^2-2t_\alpha}{\gamma_\alpha^2} {+ O\left(\left(1+\frac{t_\alpha^2}{\gamma_\alpha^2} \right)\frac{|\cdot-x_\alpha|}{r_{\alpha,\delta}}\right) \Bigg) dy}\\
&{\quad\quad\quad\quad\quad\quad +O\left( \int_{B_{x_\alpha}(r_{\alpha,\delta})} \frac{\exp(-\kappa t_\alpha)}{\mu_\alpha^2 \gamma_\alpha^4}dy\right)\,}\\
&=4\int_{B_{x_\alpha}\left(\frac{r_{\alpha,\delta}}{\mu_\alpha} \right)}\exp(-2T_0)\left(1+\frac{(2S_0+T_0^2-2T_0{)}}{\gamma_\alpha^2} \right) dz+O\left(\frac{1}{\gamma_\alpha^4}+\frac{\mu_\alpha}{r_{\alpha,\delta}} \right)\,{.}\\
\end{split}
\end{equation}
{Note that the term of order $\gamma_\alpha^{-2}$ in \eqref{IntFToProve3} vanishes since, by \eqref{EqS} and \eqref{AsymptExpan}, we get}
$$A_0=\int_{\mathbb{R}^2}\Delta S_0 dz=\int_{\mathbb{R}^2}4\exp(-2T_0)T_0 dz\,.$$ Then \eqref{RalphaandMualpha} and \eqref{IntFToProve3} imply
{\begin{equation}\label{FToProve3}
\Lambda_\alpha \int_{B_{x_\alpha}(r_{\alpha,\delta})} v_\alpha^2 \exp(v_\alpha^2) dy = 4\pi + O\left(\frac{1}{\gamma_\alpha^4}\right). 
\end{equation}}
 Independently, we compute 
 \begin{equation}\label{FToProve2}
 \begin{split}
 &\int_{B_{x_\alpha}(r_{\alpha,\delta})^c} v_\alpha^2 \exp(v_\alpha^2) dy\\
 &=O\Bigg(\int_{B_{x_\alpha}(r_{\alpha,\delta})^c} \left(\|v_\alpha\|_p^2+\frac{1}{\gamma_\alpha^2}\left(\log\frac{C}{|x_\alpha-y|}\right)^2 \right)\times\\
&\quad\quad\quad\quad\quad\quad\exp\left(\left[\frac{1+o(1)}{\gamma_\alpha}\log\frac{1}{|x_\alpha-y|^2}+o(1)\right]^2 \right) dy  \Bigg)\,,\\
&=O\left(\|v_\alpha\|_p^2+\frac{1}{\gamma_\alpha^2} \right).
\end{split}
 \end{equation}
arguing as in \eqref{Inter9}, using  \eqref{VTo0}, \eqref{C0Theory}, \eqref{Inter10} and the dominated convergence theorem. At last, we easily get from \eqref{C0Theory}  {and \eqref{VpLarge}} that
 \begin{equation}\label{L2Norm}
 \int_\Omega v_\alpha^2 dy=\|v_\alpha\|_p^2\int_\Omega v^2 dy+o(\|v_\alpha\|_p^2){.}
 \end{equation}
Then, we conclude from \eqref{FToProve3}-\eqref{L2Norm} with \eqref{AccurateLambdaEstim} and \eqref{VpLarge} that \eqref{FToProve} holds true, which concludes the proof of Proposition \ref{MainProp}.
\end{proof}

\nocite{MalchMartJEMS}

\end{document}